\documentclass[preprint,12pt]{amsart}

\usepackage[margin=1.5in]{geometry}

\usepackage[utf8]{inputenc}
\usepackage[T1]{fontenc}
\usepackage{lmodern}
\usepackage[english]{babel}
\usepackage{amsfonts,amsmath,mathrsfs,amssymb,amsthm,amscd,latexsym,amstext,amsxtra}
\usepackage{graphicx}
\usepackage[all]{xy}
\usepackage{enumerate}
\usepackage{booktabs}
\usepackage{caption}
\usepackage{csquotes}

\usepackage[style=numeric, backend=biber]{biblatex}

\usepackage{tikz}
\usetikzlibrary{cd}

\theoremstyle{plain}

\newtheorem{theorem}{Theorem}[section]

\newtheorem{lemma}[theorem]{Lemma}

\theoremstyle{definition}
\newtheorem{definition}[theorem]{Definition}

\newtheorem{remark}[theorem]{Remark}

\newcommand{\OO}{\mathcal{O}}

\newcommand{\set}[1]{\left\lbrace #1 \right\rbrace}

\makeatletter
\def\ps@pprintTitle{%
  \let\@oddhead\@empty{}
  \let\@evenhead\@empty{}
  \let\@oddfoot\@empty{}
  \let\@evenfoot\@oddfoot{}
}
\makeatother

\title{A Curve of Secants to the Kummer Variety from Degenerate Points}
\author{José Alejandro Aburto}
\thanks{Supported by ANID, Doctorado Nacional 2018, folio 21180924}
\email{jose.aburto@ug.uchile.cl}


\begin{document}

\maketitle

\begin{abstract}
We prove that, under certain geometric conditions, that only \(m-1\) different non-degenerate \((m+2)\)-secant 
\(m\)-planes
plus one degenerate \((m+2)\)-secant \(m\)-plane to the Kummer variety implies the existence of a curve of 
${(m+2)}$-secants to the Kummer variety.
This is done by constructing a set of equations in terms of theta functions from the germ of a curve on the described points.
The relation between those equations allows to proceed by induction to get the entire desired curve since the first of them is equivalent to the hypothesis that we ask.
\end{abstract}

\section{Introduction}
Let \(\left(X,\Theta\right)\) be a principally polarized complex abelian variety.

Krichever, in \cite{Krichever2010} and \cite{KricheverFlexSolved}, showed that the existence
of one trisecant line to the Kummer variety of \(X\) implies that \(X\) is a Jacobian variety.
This was established using analytic methods involving integrable systems and wave functions.

Krichever thereby proved the Welters' conjecture \cite{Welters1984} in full. 
Following Welters' work, several attempts were made to prove the conjecture using purely
algebro-geometric methods.
One particular case, namely the flex case, was settled by Debarre \cite{DebarreTrisecants1},
who proved that the conjecture holds under certain geometric conditions.

The purpose of this article is to generalize the approach introduced by Debarre.
We prove that the existence of 
a curve of $(m+2)$-secant $m$-planes to the Kummer variety can be achieved by the existence of a finite number of
$(m+2)$-secant \(m\)-planes to the Kummer variety of $X$.
More precisely, by one degenerate \((m+2)\)-secant \(m\)-plane plus \((m-1)\) non-degenerate \(m+2\)-secant \(m\)-planes.

The technique introduced by Debarre \cite{DebarreTrisecants1} for trisecants relies on the Koszul complex
in an essential way. Here we extend this approach to arbitrary \((m+2)\)-secants \(m\)-planes by iterating
the Koszul complex. Although the final construction is straightforward in retrospect,
its implementation involves a non-trivial inductive procedure that reveals the underlying hierarchy of
secant conditions.

Unfortunately, our hypothesis is also the generalization of that of Debarre;
we have been unable to circumvent it. Note that for \(m=1\), the hypothesis coincides verbatim with Debarre's.

 \section{Preliminaries}\label{sec:preliminaries}
Let \((X,\lambda)\) be an indecomposable principally polarized abelian variety of dimension $g$,
and let $\Theta$ be a symmetric representative of the polarization $\lambda$.
We can take a basis \(\{\theta\}\) of \(H^0(X,\mathcal{O}_X(\Theta))\) and a basis
\(\set{\theta_0,\ldots,\theta_{N}}\) (with $N=2^g-1$) of the $2^g$-dimensional space
\(H^0(X,\mathcal{O}_X(2\Theta))\) which satisfies the addition formula: 
	\begin{equation*}
		\theta(z+w)\theta(z-w) = \sum_{j=0}^N \theta_j(z)\theta_j(w),
	\end{equation*}
for all $z,w\in X$. 

Since symmetric representatives $\Theta$ of the polarization $\lambda$ differ by translations by points of order 2,
 we have that the linear system $\left|2\Theta\right|$ is independent of the choice of $\Theta$.
 It defines a morphism $K:X\to\left|2\Theta\right|$ called the \emph{Kummer morphism}.
This morphism can be given by \({K(x)=[\theta_0(x):\cdots:\theta_N(x)]}\).

Let $x\in X$ be any point, we write $\Theta_x$ for the divisor $\Theta+x$ and $\theta_x$ for the section $z\mapsto\theta(z-x)$ of $H^0\left(X,\OO_X(\Theta_x)\right)$.

We now define what is an \((m+2)\)-secant \(m\)-plane.

\begin{definition}
	Let \(Y\subset X\) be an Artinian scheme of length \(m+2\). Define the scheme:
		\begin{equation*}
			V_Y	=	2\left\lbrace\zeta\in X\;:\; \exists W\in\mathbb{G}\left(m, 2^g-1\right)\text{ such that } \zeta+Y\subset K^{-1}W\right\rbrace.
		\end{equation*}
	We will say that \(X\) have \((m+2)\)-secant \(m\)-planes if such a subscheme is not empty and \(K\)
    restricted to \(\zeta+Y\) is an embedding.
\end{definition}

In particular, if \(Y=\{a_1,\ldots,a_{m+2}\}\) has
reduced structure we can write

\begin{equation}
	V_Y	= 2\left\lbrace\zeta\in X\;:\; K\left(\zeta+a_1\right)\wedge\cdots\wedge K\left(\zeta+a_{m+2}\right)=0\right\rbrace.
\end{equation}
Since \(K\) is an even function 
we have that \(-a_1-a_2\in V_Y\).

We will assume that there is no \((n+2)\)-secant \(n\)-planes for \(n<m\) in what follows.
\subsection{Hierarchy}\label{subsec:hierarchy}
Our intention is to translate the condition \({\dim_{-a_1-a_2}V_Y>0}\) into an infinite set of equations.
We will look for a smooth germ with \(D_1\neq 0\).

The condition \({\dim_{-a_1-a_2}V_Y>0}\) is equivalent to the existence of a formal curve
\begin{equation*}
	2\zeta(\epsilon)=-a_1-a_2+C(\epsilon),\quad\text{with } C(\epsilon)=\sum_{j\geq 1}W^{(j)}\epsilon^j
\end{equation*}
contained in \(V_Y\) with \(W^{(j)}=\left(W^{(j)}_1,\ldots,W^{(j)}_{g}\right)\in\mathbb{C}^g\).
We define the following differential operators:

\begin{align*}
	D_j			& = \sum_{i=1}^{g} W^{(j)}_i \frac{\partial}{\partial z_i}
	\intertext{and set}
	D(\epsilon)	& = \sum_{j\geq 0}	D_j \epsilon^j.
\end{align*}
Also set,
\begin{align*}
	\Delta_s	& = \sum_{i_1+2i_2+\cdots+si_s=s}\frac{1}{i_1!\cdots i_s!} D_1^{i_1}\cdots D_s^{i_s}.
\end{align*}
Then, for a \(C^\infty\) function \(f\) in \(\mathbb{C}^g\), and a constant \(C_0\in\mathbb{C}\), one has
\begin{align}
	f\left(C_0+C(\epsilon)\right)	& = \sum_{s\geq 0} \Delta_s(f)\vert_{\zeta=C_0}\epsilon^s.
	\intertext{So that}
	e^{D(\epsilon)}	& = \sum_{s\geq 0} \Delta_s(\cdot)\vert_{\zeta=C_0} \epsilon^s.
\end{align}
Therefore, the existence of such a formal curve turns into the following equivalent relation:

\begin{equation}
	\sum_{j=1}^{m+2} \alpha_j(\epsilon) K\left(a_j+\zeta(\epsilon)\right) = 0. \label{equation:mPlanes:firstKCondition}
\end{equation}
Here \(\alpha_j(\epsilon)\), for \(1\leq j\leq m\), are relatively prime elements in
\(\mathbb{C}\left[\left[ \epsilon\right]\right]\).
Set \({u=a_1-\frac{1}{2}\left(a_1+a_2\right)}\) and \({b_j=a_{j+2}-\frac{1}{2}\left(a_1+a_2\right)}\)
for \(1\leq j\leq m\).
Now, equation (\ref{equation:mPlanes:firstKCondition}) becomes:
\begin{align}
	\alpha_1(\epsilon)K\left(u+\frac{1}{2}C(\epsilon)\right)+\alpha_2(\epsilon)K\left(-u+\frac{1}{2}C(\epsilon)\right) & \nonumber\\
	+\sum_{j=1}^m \alpha_{j+2}(\epsilon)K\left(b_j+\frac{1}{2}C(\epsilon)\right) &= 0 \label{equation:mPlanes:ReplacingThedifference}
\end{align}

Taking dot product with the left side of equation (\ref{equation:mPlanes:ReplacingThedifference}), using the Riemann addition formula we get:
\begin{align*}
&\alpha_1(\epsilon)\theta\left(z+u+\frac{1}{2}C(\epsilon)\right)\theta\left(z-u-\frac{1}{2}C(\epsilon)\right)+\alpha_2(\epsilon)\theta\left(z-u+\frac{1}{2}C(\epsilon)\right)\theta\left(z+u-\frac{1}{2}C(\epsilon)\right) \nonumber  \\
&+ \sum_{j=1}^m \alpha_{j+2}(\epsilon) \theta\left(z+b_j+\frac{1}{2}C(\epsilon)\right)\theta\left(z-b_j-\frac{1}{2}C(\epsilon)\right) = 0
\end{align*}
Write \(P(z,\epsilon)=\sum_{s\geq 0} P_s(z)\epsilon^s\) for the left side of the last equation.
Note that \(P_s\in H^0\left(X,\mathcal{O}_X(2\Theta)\right)\). We have that
\begin{equation*}
P_0 = \left(\alpha_1(0)+\alpha_2(0)\right)\theta_{u}\theta_{-u}+\sum_{j=1}^m \alpha_{j+2}(0)\theta_{b_j}\theta_{-b_j}.
\end{equation*}
Since there are no \((n+2)\)-secant \(n\)-planes, with \(n<m\), \(P_0\) vanishes if and only if \({\alpha_1(0)+\alpha_2(0)=0}\) and \({\alpha_j(0)=0}\) for any \(3\leq j\leq m+2\). Given that \(\alpha_1(\epsilon),\ldots,\alpha_{m+2}(\epsilon)\) are relatively prime, we get that \(\alpha_1(\epsilon)\) and \(\alpha_2(\epsilon)\) are units.
We may assume:
\[\alpha_1(\epsilon)=1+\sum_{i\geq 1} \alpha_{1,i} \epsilon^i,\qquad \alpha_2(\epsilon)=-1,\]
\[\alpha_j(\epsilon)=\sum_{i\geq 1} \alpha_{j,i} \epsilon^i, \text{ for } j \geq 2.\]
Observe that
\begin{equation*}
P_1=\alpha_{1,1}\theta_{u}\theta_{-u}+D_1\theta_{-u}\cdot\theta_u-D_1\theta_u\theta_{-u}+\sum_{j=1}^m \alpha_{j+2,1}\theta_{\beta_{b_j}}\theta_{-\beta_{b_j}}
\end{equation*}
where \(D_1\neq 0\) and \(2u\neq 0\). Recall that we are assuming that there are no \((n+2)\)-secants,
with \(n<m\), so \(\alpha_{j+2,1}\neq 0\) for any \(3\leq j\leq m+2\).
Allowing linear changes in the \(D_i\) operators, we may assume that \(\alpha_{m+2}(\epsilon)=\epsilon\). We state this procedure as follows:

\begin{theorem}\label{theorem:hierarchyTrickDegeneratedGeneralized}
The abelian variety \(X\) satisfies \(\dim_{-a_1-a_2}V_Y>0\) if and only if there exist complex numbers
\(\alpha_{j,i}\), with \(1\leq j\leq m+1\), \(j\neq 2\), \(i\geq 1\); and constant vector fields \(D_1\neq 0, D_2,D_3\ldots\) on \(X\) such that
the sections \(P_s\) defined before, vanish for all positive integers \(s\).
\end{theorem}

\begin{remark}
Note that \(P_s\) only depends on the corresponding \(\alpha_{j,i}\), and \(D_i\) with \(i\leq s\).
Hence, we can write \(P_s\) as:
    \begin{equation*}
	P_s = Q_s +\alpha_{1,s}\theta_u\theta_{-u}+D_s\theta_{-u}\cdot\theta_u-D_s\theta_u\cdot\theta_{-u}+\sum_{j=1}^{m-1}\alpha_{j+2,s}\theta_{b_j}\theta_{-b_j}
    \end{equation*}
	where \(Q_s\) does not depend on \(\alpha_{j,s}\) nor \(D_s\).
    Furthermore, \(Q_s\in H^0\left(X,\mathcal{O}_X(2\Theta)\right)\).
\end{remark}

\begin{lemma} \label{lemma:reducing2QsGeneralization}
Fix an integer \(s\geq 1\). Suppose that \(\Theta_{u}\cdot\Theta_{-u}\cdot\Theta_{b_1}\ldots\cdot\Theta_{b_{m-1}}\)
is a complete intersection.
Then, the section \(P_s\in H^0\left(X,\mathcal{O}_X(2\Theta)\right)\) vanishes for some choice of
\(\alpha_j,s\) and \(D_s\) for \(j\in\{1,3,4,\ldots,m+1\}\)
if and only if the section \(Q_s\in H^0\left(X,\mathcal{O}_X(2\Theta)\right)\)
vanishes when restricted to \({\Theta_{u}\cdot\Theta_{-u}\cdot\Theta_{b_1}\ldots\cdot\Theta_{b_{m-1}}}\).
\end{lemma}
Note that  \(\Theta_{u}\cdot\Theta_{-u}\cdot\Theta_{b_1}\ldots\cdot\Theta_{b_{m-1}}\) is just
a translate of \(\Theta_{a_1}\cdot\Theta_{a_2}\cdot\ldots\cdot\Theta_{a_{m+1}}\).

\begin{proof}
	Consider the ideal sheaf exact sequence
    \begin{center}
		\begin{tikzcd}[font=\small, row sep=small, column sep=small]
    0 \arrow[r] & 
    \mathcal{O}_{\Theta_u\cdot\Theta_{-u}\cdot\Theta_{b_1}\cdots \Theta_{b_{m-2}}}(-\Theta_{b_{m-1}}) \arrow[r,"\cdot\theta_{b_{m-1}}"] & 
    \mathcal{O}_{\Theta_u\cdot\Theta_{-u}\cdot\Theta_{b_1}\cdots \Theta_{b_{m-2}}} \arrow[r] & 
    \mathcal{O}_{\Theta_u\cdot\Theta_{-u}\cdot\Theta_{b_1}\cdots \Theta_{b_{m-1}}} \arrow[r] & 
    0
\end{tikzcd}
\end{center}
    we apply \(\otimes\mathcal{O}_{\Theta_u\cdot\Theta_{-u}\cdot\Theta_{b_1}\cdots \Theta_{b_{m-2}}}\left(2\Theta\right)\), and using the Theorem of the Square we get \({2\Theta\sim \Theta_{b_{m-1}}+\Theta_{-b_{m-1}}}\), so that
	\begin{center}
		\begin{tikzcd}[font=\small, row sep=small, column sep=small]
	0\arrow[r] &
	\mathcal{O}_{\Theta_u\cdot\Theta_{-u}\cdot\Theta_{b_1}\cdots \Theta_{b_{m-2}}}\left(\Theta_{-b_{m-1}}\right)\arrow[r,"\cdot\theta_{b_{m-1}}"] &
	 \mathcal{O}_{\Theta_u\cdot\Theta_{-u}\cdot\Theta_{b_1}\cdots \Theta_{b_{m-2}}}\left(2\Theta\right)\arrow[r] &
	 \mathcal{O}_{\Theta_u\cdot\Theta_{-u}\cdot\Theta_{b_1}\cdots \Theta_{b_{m-1}}}\left(2\Theta\right)\arrow[r] & 0
		\end{tikzcd}
	\end{center}
	Then, using the cohomology of this exact sequence we get that \(Q_s\) vanishes on
	\({\Theta_u\cdot\Theta_{-u}\cdot\Theta_{b_1}\cdots \Theta_{b_{m-1}}}\) if and only if there exists a complex number \(\alpha_{m+1,s}\) such that
	\begin{equation*}
		\left.\left(Q_s+\alpha_{m+1,s} \theta_{b_{m-1}}\theta_{-b_{m-1}}\right)\right\vert_{\Theta_u\cdot\Theta_{-u}\cdot\Theta_{b_1}\cdots \Theta_{b_{m-2}}}=0.
	\end{equation*}
	We repeat the process until we get that
	\begin{equation*}
		\left.\left(Q_s+\sum_{j=1}^{m-1}\alpha_{j+2,s}\theta_{b_j}\theta_{-b_j}\right)\right\vert_{\Theta_u\cdot\Theta_{-u}}=0.
	\end{equation*}
	Now, as before we consider the exact sequence
	\begin{center}
		\begin{tikzcd}
			0 \arrow[r] &
			\mathcal{O}_{\Theta_u}\left(-\Theta_{-u}\right) \arrow[r,"\cdot\theta_{-u}"] &
			\mathcal{O}_{\Theta_u} \arrow[r] &
			\mathcal{O}_{\Theta_u\cdot\Theta_{-u}}\arrow[r] &
			0,
		\end{tikzcd}
	\end{center}
we apply \(\otimes\mathcal{O}_{\Theta_u}\left(2\Theta\right)\), and using the Theorem of the Square we get \({2\Theta\sim \Theta_u+\Theta_{-u}}\), so that
	\begin{center}
		\begin{tikzcd}
			0 \arrow[r] &
			\mathcal{O}_{\Theta_u}\left(\Theta_{u}\right) \arrow[r,"\cdot\theta_{-u}"] &
			\mathcal{O}_{\Theta_u}\left(2\Theta\right) \arrow[r] &
			\mathcal{O}_{\Theta_u\cdot\Theta_{-u}}\left(2\Theta\right)\arrow[r] &
			0.
		\end{tikzcd}
	\end{center}
For what is next note that \(H^0\left(\Theta_u,\mathcal{O}_{\Theta_u}\left(\Theta_{u}\right)\right)\) is generated by \(D\theta_u\), where \(D\) is a constant vector field.

Taking the long exact sequence in cohomology
\begin{center}
	\begin{tikzcd}[row sep=small, column sep=small]
		0 \arrow[r] &
		H^0\left(\Theta_u,\mathcal{O}_{\Theta_u}\left(\Theta_{u}\right)\right) \arrow[r,"\cdot\theta_{-u}"]\arrow[d,equals] &
		H^0\left(\Theta_u,\mathcal{O}_{\Theta_u}\left(2\Theta\right)\right) \arrow[r] &
		H^0\left(\Theta_u,\mathcal{O}_{\Theta_u\cdot\Theta_{-u}}\left(2\Theta\right)\right)\arrow[r] &
		\cdots \\
		& \left\langle D\theta_u\right\rangle & & &
	\end{tikzcd}
\end{center}
implies that \(\left(Q_s+\sum_{j=1}^{m-1}\alpha_{j+2,s}\theta_{b_j}\theta_{-b_j}\right)\) vanishes on \(\Theta_u\cdot\Theta_{-u}\) if and only if
\begin{equation*}
	\left.\left(Q_s+\sum_{j=1}^{m-1}\alpha_{j+2,s}\theta_{b_j}\theta_{-b_j}-D\theta_u\cdot\theta_{-u}\right)\right\vert_{\Theta_u}=0.
\end{equation*}

The section
\begin{equation*}
	\left(Q_s+\sum_{j=1}^{m-1}\alpha_{j+2,s}\theta_{b_j}\theta_{-b_j}-D\theta_u\cdot\theta_{-u}+D\theta_{-u}\cdot\theta_u\right)\in H^0\left(X,\mathcal{O}_{X}\left(2\Theta\right)\right)
\end{equation*}
vanishes on \(\Theta_u\) and one concludes by the following the same steps as before.
We get the short exact sequence of sheaves
\begin{center}
	\begin{tikzcd}
		0 \arrow[r] &
		\mathcal{O}_{X}\left(\Theta_{u}\right) \arrow[r] &
		\mathcal{O}_{X}\left(2\Theta\right) \arrow[r] &
		\mathcal{O}_{\Theta_u}\left(2\Theta\right)\arrow[r] &
		0.
	\end{tikzcd}
\end{center}
and consider its exact sequence in cohomology
\begin{center}
	\begin{tikzcd}
		0 \arrow[r] &
		H^0\left(X,\mathcal{O}_{X}\left(\Theta_{u}\right)\right) \arrow[r] &
		H^0\left(X,\mathcal{O}_{X}\left(2\Theta\right)\right) \arrow[r] &
		H^0\left(X,\mathcal{O}_{\Theta_u}\left(2\Theta\right)\right)\arrow[r] &
		0
	\end{tikzcd}
\end{center}
to conclude that there exists a constant \(\alpha_{1,s}\in\mathbb{C}\) such that
\begin{equation*}
	Q_s +\alpha_{1,s}\theta_u\theta_{-u}+D_s\theta_{-u}\cdot\theta_u-D_s\theta_u\cdot\theta_{-u}+\sum_{j=1}^{m-1}\alpha_{j+2,s}\theta_{b_j}\theta_{-b_j}=0.
\end{equation*}
\end{proof}

\subsection{Finite secants implies the existence of a curve}\label{subsec:finiteToCurve}
Now we show our main result. Note that having a degenerate \((m+2)\)-secant \(m\)-plane is equivalent 
to assuming that \(P_1=0\).
In this case, a degenerate \((m+2)\)-secant \(m\)-plane consists of \((m+1)\)-points of \(X\) whose image through the Kummer
morphism \(K\) lie on a \(m\)-plane of \(K(X)\) tangent with multiplicity \(2\) at one of them.

For what follows, consider the reduced subscheme \(Y=\{a_1,\ldots,a_{m+2}\}\) of \(X\).
Due to the equivalencies from the previous sections, we can achieve the following result from geometric properties of \(X\)
and the points \(u,b_1,\ldots,b_m\) of \(X\) defined in the previous section.

\begin{theorem}\label{theorem:mDifferent}
	Let \((X,\lambda)\) be an indecomposable principally polarized abelian variety of dimension \(g\) that has no \((n+2)\)-secant \(n\)-planes for \(n<m\),
    let \(\Theta\) be a symmetric representative of the polarization \(\lambda\).
	Suppose that \(\Theta_{u}\cdot\Theta_{-u}\cdot\Theta_{b_1}\ldots\cdot\Theta_{b_{m-1}}\) is a reduced complete intersection,
    and suppose that one has the following \((m+2)\)-secant \(m\)-planes:
 \begin{itemize}
	\item[\(\bullet\)] {\( K(u), K(b_1),\ldots,K(b_m) \) lie on an \(m\)-plane, tangent at \(K(u)\) and}
	\item[\(\bullet\)] {\( K\left(u-\tfrac{1}{2}(b_m+\mu)\right), K\left(-u-\tfrac{1}{2}(b_m+\mu)\right),\cdots, K\left(b_m-\tfrac{1}{2}(b_m+\mu)\right)\)}
	lie on an \(m\)-plane, for \( \mu\in\left\lbrace -b_1,\ldots,-b_{m-1}\right\rbrace\).
  \end{itemize}
  Then, \(\dim {V_Y}>0\).
\end{theorem}

Note that when the intersection of the $m+1$ translates of the theta divisor is empty,
the additional hypothesis is automatically satisfied, since any section vanishes on the empty scheme.

\begin{proof}
	By \ref{theorem:hierarchyTrickDegeneratedGeneralized} and \ref{lemma:reducing2QsGeneralization} it is enough to prove that for any integer \(s\geq 2\) one has that
		\begin{equation*}
		P_1=\cdots=P_{s-1}=0 \Rightarrow P_s \text{ (or \(Q_s\) vanishes on } \Theta_{u}\cdot\Theta_{-u}\cdot\Theta_{b_1}\ldots\cdot\Theta_{b_{m-1}}.
		\end{equation*}
	Set
		\begin{equation*}
		R(z,\epsilon)=\sum_{s\geq 0}R_s(z) \epsilon^s := P\left(z+\frac{1}{2}C(\epsilon),\epsilon\right).
		\end{equation*}
	Assume that \(P_1=\cdots=P_{s-1}=0\), then \({R_1=\cdots=R_{s-1}=0}\) and that \(P_s=R_s\). Now, we have that:
	\begin{align*}
		R(z,\epsilon)	& = P\left(z+\tfrac{1}{2}C(\epsilon)\right) \\
					& = e^{\tfrac{1}{2} \sum_{r\geq 0} \Delta_r} P(z,\epsilon) \\
					& = \sum_{s\geq 0}\epsilon^s \sum_{j=0}^s \Delta_j P_{s-j}.
	\end{align*}
	Hence, we will show that \(R_s\) vanishes on \(G=\Theta_{u}\cdot\Theta_{-u}\cdot\Theta_{b_1}\ldots\cdot\Theta_{b_{m-1}}\).
	\\ We have that
		\begin{equation} \label{equation:R_sRestrictedToG}
			\left. R_s \right|_G = \Delta_{s-1}\theta_{-{b_m}}\cdot\theta_{b_m}.
		\end{equation}
	On the other hand, we set
	\begin{equation*}
	T(z,\epsilon)=\sum_{s\geq 0}T_s(z) \epsilon^s := P\left(z-\frac{1}{2}C(\epsilon),\epsilon\right),
	\end{equation*}
	and \(e^{-D(\epsilon)}=\sum_{s\geq 0}\Delta^{-}_s\).

	Similarly to \(P(z,\epsilon)\), under the vanishing of \(P_1,\ldots,P_{s-1}\) by hypothesis, we get that \(T_s=P_s=R_s\).
	In particular,
	\begin{equation}
		\left. T_s \right|_G = \Delta_{s-1}^{-}\theta_{b_m}\theta_{-b_m}+\sum_{j=1}^{m-1}\sum_{l=1}^s \alpha_{j+2,l} \theta_{-b_j}\cdot \left(\Delta_{s-j}^{-}\theta_{b_j}\right),
	\end{equation}
	Hence, \(R_s^2=R_sT_s\) so that we compute
	\begin{align*}
		\left. R_s^2 \right|_G 	& =  \left(\Delta_{s-1}^{-}\theta_{b_m}\right)\cdot\theta_{b_m}\theta_{-b_m}\cdot \left(\Delta_{s-1}\theta_{-{b_m}}\right) \\
								& +  \theta_{b_m}\theta_{-b_j}\cdot\sum_{j=1}^{m-1}\sum_{l=1}^s \alpha_{j+2,l} \left(\Delta_{s-j}^{-}\theta_{b_j}\right)\cdot(\Delta_{s-1}\theta_{-{b_m}})
	\end{align*}
	which is zero by the \(m-1\) different non-degenerate \((m+2)\)-secant \(m\)-planes
    and one degenerate \((m+2)\)-secant \(m\)-plane.
	Since the intersection is reduced, \(R_s=0\), and we conclude the proof.
\end{proof}

\section{Further work}\label{sec:furtherWork}
Following the strategy introduced by Debarre \cite{DebarreTrisecants1},
we generalize it here to arbitrary \(m\). When \(m=1\) our result recovers exactly Debarre’s theorem.
The remaining types of \((m+2)\)-secant \(m\)-plane (whose classification grows combinatorially with \(m\) are left open;
our inductive procedure in Lemma \ref{lemma:reducing2QsGeneralization} via iterated Koszul complexes provides a natural template for addressing
them in future work.

It remains open to study the other cases for an Artinian scheme on length \(m\),
and to remove or get a better understanding of the condition of reducedness of the intersection of
the Theta divisors that we used.

\printbibliography{}

@article{Welters1984,
AUTHOR = {Welters, G. E.},
 TITLE = {A criterion for {J}acobi varieties},
JOURNAL = {Ann. of Math. (2)},
FJOURNAL = {Annals of Mathematics. Second Series},
VOLUME = {120},
  YEAR = {1984},
NUMBER = {3},
 PAGES = {497--504},
  ISSN = {0003-486X},
MRCLASS = {14H40 (14K10)},
MRNUMBER = {769160},
   DOI = {10.2307/1971084},
}

@article{Krichever2010,
AUTHOR = {Krichever, Igor},
 TITLE = {Characterizing {J}acobians via trisecants of the {K}ummer
    variety},
JOURNAL = {Ann. of Math. (2)},
FJOURNAL = {Annals of Mathematics. Second Series},
VOLUME = {172},
  YEAR = {2010},
NUMBER = {1},
 PAGES = {485--516},
  ISSN = {0003-486X},
MRCLASS = {14K05 (14H40 14H42 14H70 37K20)},
MRNUMBER = {2680424},
   DOI = {10.4007/annals.2010.172.485},
}

@incollection{KricheverFlexSolved,
AUTHOR = {Krichever, I.},
 TITLE = {Integrable linear equations and the {R}iemann-{S}chottky
    problem},
BOOKTITLE = {Algebraic geometry and number theory},
SERIES = {Progr. Math.},
VOLUME = {253},
 PAGES = {497--514},
PUBLISHER = {Birkh\"{a}user Boston, Boston, MA},
  YEAR = {2006},
MRCLASS = {14H70 (14H42 37K20)},
MRNUMBER = {2263198},
   DOI = {10.1007/978-0-8176-4532-8\_8},
}

@article{DebarreTrisecants1,
AUTHOR = {Debarre, Olivier},
 TITLE = {Trisecant lines and {J}acobians},
JOURNAL = {J. Algebraic Geom.},
FJOURNAL = {Journal of Algebraic Geometry},
VOLUME = {1},
  YEAR = {1992},
NUMBER = {1},
 PAGES = {5--14},
  ISSN = {1056-3911},
MRCLASS = {14H42 (14K10)},
MRNUMBER = {1129836},
}
\end{document}